\documentclass{amsart}
\usepackage{delimset, amssymb, enumitem, caption}
\setlist{itemsep=4pt, topsep=4pt, leftmargin=17pt, listparindent=11pt}
\usepackage[hyphens]{url}
\usepackage[colorlinks, citecolor=cyan, linkcolor=magenta, pagebackref, urlcolor=brown, ocgcolorlinks]{hyperref}
\usepackage{tikz, tkz-berge}
\usepackage{genyoungtabtikz}
\SetVertexNoLabel
\tikzstyle{EdgeStyle}=[very thick, color=cyan]
\tikzstyle{VertexStyle}=[shape=circle, inner sep=6pt, ball color=magenta!25!white]
\tikzset{edge/.style={very thick, color=cyan}}
\tikzset{ball/.style={shape=circle, inner sep=9pt, ball color=magenta!25!white}}
\tikzset{OuterBoundary/.style={line width=1.2pt}} 
\tikzset{RimHook/.style={cyan, line width=1.2pt, rounded corners=2pt}}
\tikzset{RHiball/.style={ball color=magenta!50!white}}
\tikzset{RHeball/.style={ball color=magenta!25!white}}
\tikzset{apball/.style={ball color=black}}
\tikzset{eedge/.style={very thick, color=cyan}}
\tikzset{vvertex/.style={ball color=magenta!25!white}}
\tikzset{Brick/.style={line width=1.2pt, rounded corners=5pt}}
\tikzset{ImaginaryLine/.style={color=gray}}
\tikzset{ImagLine/.style={dashed, color=cyan, line width=1.5}}
\YFrench
\Yboxdim{1cm}

\newcommand\SA{\mathrm{SA}}

\usepackage[capitalize]{cleveref}
\crefname{ineq}{Ineq.}{Ineqs.}
\Crefname{ineq}{Inequality}{Inequalities}
\creflabelformat{ineq}{~\upshape(#2#1#3)}
\crefname{itm}{}{}
\creflabelformat{itm}{~\upshape(#2#1#3)}
\crefname{rl}{Relation}{Relations}
\creflabelformat{rl}{~\upshape(#2#1#3)}
\crefname{conjecture}{Conjecture}{Conjectures}

\newtheorem{theorem}{Theorem}[section]
\newtheorem{corollary}[theorem]{Corollary}
\newtheorem{lemma}[theorem]{Lemma}
\newtheorem{conjecture}[theorem]{Conjecture}
\newtheorem{proposition}[theorem]{Proposition}

\theoremstyle{definition}

\theoremstyle{remark}

\numberwithin{equation}{section}
\numberwithin{figure}{section}
\numberwithin{table}{section}

\usepackage{xcolor}

\allowbreak
\allowdisplaybreaks

\usepackage[numbers, sort&compress, nonamebreak, merge, elide, longnamesfirst]{natbib}

\usepackage{mathtools}%%to use floor*, ceil*, and abs* to produce the \left\right aotumatically.

\DeclarePairedDelimiter\floor{\lfloor}{\rfloor}

\title[]{The $e$-positivity and Schur positivity of the chromatic symmetric functions of some trees}

\author[D.G.L.~Wang]{David G.L. Wang}
\address[David G.L. Wang]{School of Mathematics and Statistics \& Beijing Key Laboratory on MCAACI, Beijing Institute of Technology; Key Laboratory of Mathematical Theory and Computation in Information Security, Ministry of Industry and Information Technology, 102400 Beijing, R.R. China}
\email{glw@bit.edu.cn}

\author[M.M.Y. Wang]{Monica M.Y. Wang}
\address{
School of Mathematics and Statistics, Beijing Institute of 
Technology, 102400 Beijing, P. R. China}
\email{mengyu919@bit.edu.cn}



\thanks{This paper is supported by the General Program of National Natural 
Science Foundation of China (Grant No.~12171034) and 
the Fundamental Research Funds for the Central Universities in China (Grant No.~2021CX11012).}

\begin{document}
\bibliographystyle{abbrvnat}
\maketitle
\begin{abstract}
We investigate the $e$-positivity and Schur positivity of the chromatic symmetric functions of some spider graphs with three legs. We obtain the positivity classification of all broom graphs and that of most double broom graphs. The methods involve extracting particular $e$-coefficients of the chromatic symmetric function of these graphs with the aid of Orellana and Scott's triple-deletion property, and using the combinatorial formula of Schur coefficients by examining certain special rim hook tabloids. We also propose some conjectures on the $e$-positivity and Schur positivity of trees.
\end{abstract}

%
%05C05 Trees
%
%05C30 Enumeration in graph theory
%
%05A15 Exact enumeration problems, generating functions
%
%05C15 Coloring of graphs and hypergraphs
%
%05C25 Graphs and abstract algebra (groups, rings, fields, etc.)
%
%05C31 Graph polynomials
%
%05C60 Isomorphism problems in graph theory (reconstruction conjecture, etc.) and homomorphisms (subgraph embedding, etc.)
%
%05E05 Symmetric functions and generalizations

%\tableofcontents

\section{Introduction}

\citet{Sta95} introduced 
the {\it chromatic symmetric function}
for a simple graph $G$ as
\[
X_G
=X_G(x_1, x_2, \ldots )
=\sum_{\kappa}\prod_{v\in V(G)}\mathbf{x}_{\kappa(v)}
\]
where $\mathbf{x}=(x_1, x_2, \ldots)$ is a countable set
of indeterminates, 
and the sum is over all proper colorings~$\kappa$, 
namely
colorings such that every each color class 
is an independent set.
The chromatic polynomial~$\chi_G(k)$
counts the number of proper colorings using $k$ colors,
which is a classical graph invariant dating back to~\citet{Bir12}.
%while symmetric functions date back even further to Cauchy.
%They have been popular for centuries from algebraic geometry
%to quantum physics.
The chromatic symmetric function satisfies $X_G(1^k)=\chi_G(k)$.
The chromatic symmetric function $X_G$ is a symmetric function.
Common bases for the algebra $\Lambda(x_1,x_2,\ldots)$
of symmetric functions include
the monomial symmetric functions~$\{m_\lambda\}$,
elementary symmetric functions~$\{e_\lambda\}$,
Schur symmetric functions~$\{s_\lambda\}$ and so on,
see \cite[Chapter 7]{Sta99B}.
For any basis~$\{b_{\lambda}\}$ of $\Lambda(x_1,x_2,\ldots)$, 
the graph $G$ is said to be {\it $b$-positive} 
if the expansion of $X_G$ in~$b_{\lambda}$ has only nonnegative coefficients.

%We first present the motivation of studying the $e$-positivity,
%then that of studying the Schur positivity, 
%and that for why we study them for trees.

A strong motivation of studying the positivity of chromatic symmetric functions
is Stanley and Stembridge's conjecture posed in 1993, see \cref{conj:ep:cf-inc}.
The smallest connected non-$e$-positive graph is the claw, 
whose chromatic symmetric function is 
\[
X_{\text{claw}}
=4e_4+5e_{31}-2e_{2^2}+e_{21^2}
=s_{31}-s_{2^2}+5s_{21^2}+8s_{1^4}.
\] 
A graph is \emph{claw-free} if it does not contain 
an induced subgraph which is isomorphic to the claw.
The \emph{incomparability graph} of a poset $P$ is the graph with vertex set $P$, in which two elements are adjacent if and only if they are incomparable.

\begin{conjecture}[\citeauthor{SS93}]\label{conj:ep:cf-inc}
Any claw-free incomparability graph is $e$-positive.
\end{conjecture}

\citet{Wol97D} provided a powerful criterion that
any connected $e$-positive graph has a connected partition of every type,
where a connected partition of a graph $G$ is a partition $\{V_1, \dotsc, V_k\}$
of $V(G)$ such that each induced subgraph $G[V_i]$ is connected. 
Many graph classes are shown to be $e$-positive, including
complete graphs, paths, cycles, 
triad-free graphs,
generalized bull graphs, 
(claw, $K_3$)-free graphs, (claw, co-$P_3$)-free graphs,
$F$-free unit interval graphs for all 4-vertex graphs $F$ except the co-diamand, $K_4$, $4K_1$ and $2K_2$,
(claw, paw)-free graphs, (claw, co-paw)-free graphs, 
(claw, diamond, co-diamond)-free graphs,
(claw, triangle)-free graphs, (claw, co-$P_3$)-free graphs, 
(claw, co-diamond, $2K_2$)-free graphs,
$2K_2$-free unit interval graphs,
$K$-chains, lollipop graphs, triangular ladders,
(claw, co-claw)-free graph except the net,
Ferrers graphs;
see \cite{Cv16, Sta95, CH18, CH19, FHM19, Tsu18, HHT19, Dah19, Dv18, GS01, Ev04,  LY21}.
Graphs that are proved not to be $e$-positive include the dart,
generalized nets,
saltire graphs $\SA_{n,n}$,
augmented saltire graphs $AS_{n,n}$ and $AS_{n,n+1}$,
triangular tower graphs $TT_{n,n,n}$;
see \cite{DFv20, DSv20, FKKMMT20, FHM19}.
In 2020 \citet{DFv20} gave an infinite number of families of non-$e$-positive graphs that are not contractible to the claw. Moreover, one such family is additionally claw-free, thus establishing that the $e$-positivity is in general not dependent on the existence of an induced claw or of a contraction to a claw.

A second motivation of studying the positivity of chromatic symmetric functions
is the close relationship between the Schur postivity and representation theory.
The Schur functions are considered to be 
the most important basis of the algebra $\Lambda(x_1,x_2,\dots)$
from several perspectives, see~\citet{Mac95B,Mac15B,Sag01B} and \citet{Sta81,Sta99B}.
\citet{Gas96P} showed that any claw-free incomparability graph is Schur positive.
Every $e$-positive graph is Schur positive
since the $s_\lambda$-coefficient in $e_\mu$ 
is the Kostka number $K_{\lambda',\mu}$ which is nonnegative,
see~\citet[Exercise 2.12]{MR15B}.
A leading conjecture in this direction is due to \citet{Gas99} and \citet{Sta98}.

\begin{conjecture}[\citeauthor{Gas99,Sta98}]
Every claw-free graph is Schur positive.
\end{conjecture}

\citet[Proposition 1.5]{Sta98} proved that
the set of types of stable partitions of any Schur positive $n$-vertex graph 
is an order ideal of the poset of integer partitions of $n$ with respect to the dominance order.
The authors \cite{WW20} gave a combinatorial formula for the Schur coefficients of chromatic symmetric functions.
Graphs that are shown to be Schur positive include
tadpole graphs,
the graphs obtained from two cycles $C$ and $C'$
by adding a path linking a vertex on $C$
and a vertex on $C'$,
claw-free incomparability graphs, 
edge 2-colorable hyperforests,
the incomparability graph of the natural unit interval order;
see \cite{Gas96P, Gas99, SW16, LLWY20}.
Graphs that are proved not to be Schur positive include
connected unbalanced bipartite graphs and 
the complete bipartite graphs $K_{m,n}$ with $m,n\ge 3$,
see \cite{WW20} for more graphs that are not Schur positive.
\citet{Kal15} confirmed the positivity of the $s_\lambda$ coefficients when $\lambda$ is of a hook shape.

In this paper, we concentrate on the chromatic symmetric functions of trees.
This is not only for the simplicity of trees as a particular graph class,
but also for the following major conjecture in this field,
which is called \emph{Stanley’s isomorphism conjecture} by \citet{LS19}
and the \emph{tree isomorphism conjecture} by \citet{CS20}.
\begin{conjecture}\label{conj:distinguishability}
The chromatic symmetric function distinguishes trees.
\end{conjecture}
In fact, \cref{conj:distinguishability}
was inspired by \citet[Page 170]{Sta95}'s remark
``We do not know whether~$X_G$ distinguishes trees''. 
See \cite{HJ19,Mor05M,Fou03B,MMW08,AZ14,FKKMMT20,Gol78,Tsu18,OS14}
for its research progress.

The problem of determining whether a given tree is $e$-positive and whether it is Schur positive also received attention.
\citet{DSv20} conjectured the existence of an $n$-vertex Schur positive tree
of maximum degree $\floor{n/2}$,
which is disproved by \citet{RS20} with a counterexample.
They \cite{DSv20} also proved 
that any $n$-vertex $e$-positive tree has degree at most 
$\log_2{n}$, and further conjectured the maximum degree of any $e$-positive tree to be 3.
\begin{conjecture}[\citeauthor{DSv20}]\label{conj:tree4:ne}
Any tree with a vertex of degree at least $4$ is not $e$-positive.
\end{conjecture}

\Citet{Zhe20X} obtained a breakthrough towards \cref{conj:tree4:ne} by proving
that any tree with a vertex of degree at least $6$ is not $e$-positive. 

A particular class of trees, the spiders, 
plays an essential role in the study of $e$-positivity of graphs.
A \emph{spider} is a tree consisting of some paths 
with one endpoint on each path identified. Precisely speaking,
for any partition 
\[
\lambda=\lambda_1\dotsm\lambda_d\vdash n-1
\]
with $d\ge 3$,
the spider $S(\lambda)$ is the $n$-vertex tree consisting 
of the paths $P_{1+\lambda_1}$, $\dots$, $P_{1+\lambda_d}$ such that
all of them share a common endpoint of degree $d$.
\citet[Lemma 13]{DSv20} showed that 
if a connected graph $G$ has a connected partition of type~$\mu$,
then the spider $S(\lambda)$ 
has a connected partition of type~$\mu$,
where $\lambda$ is the partition consisting 
of the sizes of connected components that are obtained by removing a vertex of degree at least~3 from~$G$.
Therefore, the $e$-positivity of a general graph implies
the $e$-positivity of certain spider
in view of \citeauthor{Wol97D}'s criterion.

This paper is organized as follows. 
In \cref{sec:preliminary} we give an overview for necessary notion and notation,
as well as known results in the study of graph positivities that will be of use in 
the subsequent sections.
\Cref{sec:spider:ab1,sec:spider:ab2} are devoted to the positivity of spiders $S(a,b,1)$ and $S(a,b,2)$, respectively.
We obtain some bounds of $a$ in terms of $b$ for the $e$-positivity of 
these spiders, and conjecture the Schur positivity of these spiders.
In \cref{sec:broom}, 
we obtain the positivity classification of all broom graphs and the positivity classification of most double broom graphs.
We end this paper with a conjecture that completes the 
positivity classification of double broom graphs, see \cref{conj:Schur:SPS:22}.

The chromatic symmetric functions of explicit graphs on a small number of vertices
in this paper are computed by using Russell's program~\cite{Rus19W}.

\section{Preliminaries}\label{sec:preliminary}
Let $n$ be a positive integer.
A \emph{composition} $\kappa$ of $n$ 
is a sequence $(\kappa_1,\dotsc,\kappa_\ell)$ of integers that sum to~$n$.
We write 
\[
\kappa!=\prod_{i\ge1}\kappa_i!
\quad\text{and}\quad
\kappa^!=\prod_{i\ge 1}k_i!,
\]
where $k_i$ is the number of occurrences of the part $i$ in $\kappa$.
An \emph{integer partition} $\lambda$ of~$n$ is a composition 
$(\lambda_1,\dots,\lambda_{\ell})$ of $n$
in non-increasing order, denoted 
$\lambda\vdash n$.
It can be recast as $1^{a_1}2^{a_2}\cdots$,
where $a_i$ is the multiplicity of~$i$ in~$\lambda$.

Let $G=(V,E)$ be a graph with vertex set $V$ and edge set $E$.
The \emph{order} of~$G$ is the number $\abs{V}$ of vertices.
A \emph{partition} of~$G$ is a set partition $\rho=V_1/\dotsb/V_\ell$
of $V$. 
It is said to be a \emph{bipartition} if $\ell=2$.
A bipartition $V_1/V_2$ is \emph{balanced} if $\abs{V_1}-\abs{V_2}\in\{-1,0,1\}$.
We call the sets $V_i$ \emph{blocks} of~$\rho$.
We say that a partition $\rho$ is \emph{semi-ordered}
if for any number $m$, the blocks of order~$m$ in~$\rho$ are ordered.
A block in $\rho$ is \emph{stable} if any two vertices in the block are not adjacent by an edge.
A partition~$\rho$ is \emph{stable} if its every block is stable.
The \emph{type} of $\rho$ is the integer partition consisting of 
the block cardinalities, denoted~$\tau_\rho$.
For any composition~$\kappa$ obtained by rearranging the parts of $\tau_\rho$,
without confusion, one may say that $\rho$ is of type $\kappa$.

For any partition $\lambda=(\lambda_1,\dots,\lambda_\ell)$,
the \emph{monomial symmetric function}~$m_\lambda$ is defined by
\[
m_\lambda=\sum_{\alpha}x^\alpha,
\]
where $\alpha$ runs over all distinct permutations of $\lambda$;
the \emph{augmented monomial symmetric function} $\tilde m_{\lambda}$ is defined by $\tilde m_{\lambda}=\lambda^!m_{\lambda}$; the \emph{elementary symmetric function} $e_\lambda$ is defined to be
\[
e_\lambda=m_{1^{\lambda_1}}\dotsm m_{1^{\lambda_\ell}};
\]
the \emph{Schur function} $s_{\lambda}$ is defined by
\[
s_{\lambda}=\sum_T x^T,
\]
where $T$ ranges over all semistandard Young tableaux of shape $\lambda$,
and $x^T$ is the monomial $x_1^{i_1}x_2^{i_2}\dotsm$ such that $T$ contains exactly $i_j$ cells with entry $j$ for all $j$.

\Citet[Propositions~2.3, 2.4, 5.3 and Theorem 2.5]{Sta95}
gave some basic properties of chromatic symmetric functions.

\begin{proposition}[\citeauthor{Sta95}]\label{prop:csf:disjoint}
$X_{G\sqcup H}=X_G X_H$, where $G\sqcup H$ is the disjoint union of the graphs $G$ and~$H$. 
\end{proposition}

\begin{proposition}[\citeauthor{Sta95}]\label{prop:csf}
The chromatic symmetric function $X_G$ of a graph $G=(V,E)$ can be computed by
\[
X_G
=\sum_{\lambda\vdash \abs{V(G)}}a_\lambda\tilde m_{\lambda}
=\sum_{E'\subseteq E}(-1)^{\abs{E'}}p_{\lambda(E')}
\]
where $a_\lambda$ is the number of stable partitions of $G$ of type $\lambda$,
and $\lambda(E')$ 
is the integer partition consisting of
the component orders of the spanning subgraph~$(V,E')$.
\end{proposition}

\begin{proposition}[\citeauthor{Sta95}]\label{prop:gf:path:cycle}
The chromatic symmetric functions $X_{P_n}$ for the $n$-vertex paths $P_n$ satisfy
\[
\sum_{n\geq 0}X_{P_n}z^n
=\frac{E(z)}{F(z)}
=1+e_1z+2e_2z^2+(3e_3+e_{21})z^3+\cdots,
\]
where $E(z)=\sum_{n\ge0} e_n z^n$ and $F(z)=E(z)-zE'(z)$.
\end{proposition}

\citet[Proposition 1.3.3]{Wol97D} derived a powerful criterion for the $e$-positivity of a graph.
\begin{theorem}[\citeauthor{Wol97D}]\label{thm:Wolfgang}
Any $e$-positive graph contains a connected partition of any type.
\end{theorem}

For any basis $\{b_\lambda\}$ of the algebra $\Lambda(x_1,x_2,\dots)$
and any symmetric function $F\in \Lambda(x_1,x_2,\dots)$,
we use the notation $[b_\lambda]F$ to denote the coefficient of $b_\lambda$
in the $b$-expansion of $F$.
By \cref{prop:gf:path:cycle},
\citet[Theorem 3.2]{Wol98} exhibited explicit formulas 
for the coefficients of $e_\lambda$ of paths.

\begin{proposition}[\citeauthor{Wol98}]\label{prop:Wolfe}
Let $\lambda=1^{a_1}2^{a_2}\dotsm d^{a_d}\vdash d$. Then 
\[
[e_\lambda]X_{P_d}
=\binom{\ell}{a_1,\dots,a_d}\prod_{a_j\ge 1}(j-1)^{a_j}
+\sum_{a_i\ge 1}\binom{\ell-1}{a_1,\dots,a_i-1,\dots,a_d}(i-1)^{a_i-1}\prod_{a_j\ge 1\atop{j\ne i,\,j\ne2}}(j-1)^{a_j},
\]
where $\ell=a_1+\dots+a_d$ is the length of $\lambda$.
\end{proposition}

\Citet[Theorem 3.1, Corollaries 3.2 and 3.3]{OS14} established the beautiful \emph{triple-deletion} property as follows.

\begin{theorem}[\citeauthor{OS14}]\label{thm:rec:3del}
Let $G$ be a graph with a stable set $\{u,v,w\}$.
Write $e_1=uv$, $e_2=vw$, and $e_3=wu$.
For any set $S\subseteq \{1,2,3\}$, 
denote by $G_S$ the graph with vertex set~$V(G)$
and edge set $E(G)\cup\{e_j\colon j\in S\}$.
Then 
\[
X_{G_{12}}=X_{G_1}+X_{G_{23}}-X_{G_3}
\quad\text{and}\quad
X_{G_{123}}=X_{G_{12}}+X_{G_{23}}-X_{G_2}.
\]
\end{theorem}
\citet[Proposition 5]{Dv18} generalized this to $k$-cycles, called the \emph{$k$-deletion} property.
%The paths $P_{1+\lambda_i}$ are called the \emph{legs} of $S(\lambda)$,
%and the degree~$d$ vertex is called the \emph{center} of $S(\lambda)$.
\Citet[Lemma 18 and Theorem 30]{DSv20} gave quick criteria
for the $e$-positivity of spiders.

\begin{theorem}[\citeauthor{DSv20}]\label{thm:lambda1:d:n}
Let $\lambda=(\lambda_1,\dotsc,\lambda_d)\vdash n-1$.
If the spider $S(\lambda)$ is $e$-positive,
then $\lambda_1\ge \floor{n/2}$ and $d<\log_2{n}+1$.
\end{theorem}

Here are some contributions due to \citet[Theorem 3.4,Lemma 4.4,Theorem 5.3]{Zhe20X} 
to the $e$-positivity of spiders.

\begin{lemma}[\citeauthor{Zhe20X}]\label{lem:e:spider:mod}
Let $\lambda=(\lambda_1,\dotsc,\lambda_d)\vdash n-1$.
Let $m\in\mathbb{Z}^+$ and $R_m=1+\sum_{i=1}^d r_i$,
where~$r_i$ is the least nonnegative residue of $\lambda_i$ modulo $m$.
Suppose that $n=mq+r$, where $q,r\in\mathbb{Z}$ and $0\le r\le m-1$.
If the spider $S(\lambda)$ is $e$-positive, then we have the following.
\begin{enumerate}
\item
$R_m<2m$,
\item
if $R_m\ge m$, then $r_i\ge r$ for some $i\in[d]$.
\end{enumerate}
\end{lemma}

\begin{lemma}[\citeauthor{Zhe20X}]\label{lem:2odds}
Let $\lambda=(\lambda_1,\dots,\lambda_d)\vdash n-1$.
Suppose that
\[
\{\lambda_1,\dots,\lambda_d\}=\{2k_1+1,\,2k_2+1,\,2k_3,\,2k_4,\,\dots,\,2k_d\}
\]
as multisets, where $k_i\in\mathbb{Z}$. Then
\[
\brk[s]1{e_{32^{k_1+\dots+k_d}}}X_{S(\lambda)}
=4(k_1+k_2-k_3-\dots-k_d)+2d-1.
\]
\end{lemma}

\begin{lemma}[\citeauthor{Zhe20X}]\label{lem:rec:Sabc}
Suppose that $(a,b,c)\vdash n-1$.
Then the chromatic symmetric function of the spider $S(a,b,c)$ can be computed by
\[
X_{S(a,b,c)}
=X_{P_n}+\sum_{i=1}^c \brk1{X_{P_i}X_{P_{n-i}}-X_{P_{b+i}} X_{P_{n-b-i}}}.
\]
\end{lemma}

For Schur positivity of graphs, the authors~\cite{WW20} obtained the following results.

\begin{theorem}[\citeauthor{WW20}]\label{thm:nS:balanced}
Any Schur positive connected bipartite graph
has a balanced stable bipartition.
\end{theorem}

\begin{theorem}[\citeauthor{WW20}]\label{thm:s-in-X}
For any graph $G=(V,E)$ and any integer partition $\lambda$ of $\abs{V}$,
\begin{equation}\label{s-in-X}
[s_\lambda]X_G
=\sum_{T\in\mathcal{T}_\lambda}(-1)^{\abs{W_T}} \tilde a_{\kappa_T},
\end{equation}
where $\mathcal{T}_\lambda$ is the set 
of special rim hook tabloids $T$ of shape $\lambda$ 
such that $G$ contains a stable partition of type $\kappa_T$,
$\abs{W_T}$ is the number of rim hooks of $T$ 
that span an even number of rows,
and~$\tilde a_{\kappa}$ is the number of semi-ordered stable partitions of $G$
of type $\kappa$. 
\end{theorem}

\section{The positivity of spiders $S(a,b,1)$}\label{sec:spider:ab1}

This section is devoted to the $e$-positivity and Schur positivity of spiders $S(a,b,1)$. First of all, it is hardly true that a spider $S(a,b,c)$ with odd $b$ and odd $c$ is $e$-positive.

\begin{theorem}\label{thm:spider:eoo}
Let $(a,b,c)\vdash n-1$.
Suppose that $b$ and $c$ are odd.
If the spider $S(a,b,c)$ is $e$-positive,
then $a=b+c$.
\end{theorem}
\begin{proof}
Write $(a,b,c)=(\lambda_1,\lambda_2,\lambda_3)$.
Suppose that $\lambda_i=2k_i+1$ for $i\in\{2,3\}$.
Since $S(\lambda)$ is $e$-positive, the part $\lambda_1$ must be even.
Suppose that $\lambda_1=2k_1$.
Then $n=2(k_1+k_2+k_3)+3$. By \cref{lem:2odds}, 
\[
\brk[s]1{e_{32^{k_1+k_2+k_3}}}X_{S(\lambda)}
=4(k_2+k_3-k_1)+5
=2n-1-4\lambda_1.
\]
Since $X_{S(\lambda)}$ is $e$-positive, the non-negativity of the formula above
implies that $\lambda_1\le\floor{n/2}$. 
By \cref{thm:lambda1:d:n} we know that $\lambda_1\ge \floor{n/2}$.
Hence $\lambda_1=\floor{n/2}$, i.e., $\lambda_1=\lambda_2+\lambda_3$.
\end{proof}

Conversely, 
we do not know whether the spider $S(b+c,\,b,\,c)$ is $e$-positive.
The $e$-positivity of $S(b+1,\,b,\,1)$ was conjectured 
by Aliniaeifard, van Willigenburg, and Wang, which was a particular case of a more general conjecture, see~\citet[Conjecture 6.3]{Zhe20X}.
It is direct to verify the $e$-positivity of the spider $S(2m+2,\,2m+1,\,1)$ for $m\le 15$.

We are able to prove that $S(b+3,\,b,\,3)$ for odd $b\ge7$ are not $e$-positive.

\begin{theorem}
Let $b$ be an odd positive integer. 
The spider $S(b+3,\,b,\,3)$ is $e$-positive if and only if $b=5$.
\end{theorem}
\begin{proof}
Let $G=S(2m+4,\,2m+1,\,3)$. 
The order of $G$ is $N=4m+9$.
Suppose that $m$ is an odd positive integer. By~\cref{lem:rec:Sabc,prop:Wolfe},
\begin{align*}
&[e_{54^{m+1}}]X_G\\[-5pt]
=\ &[e_{54^{m+1}}]X_{P_{4m+9}}+[e_{54^{m+1}}]\sum_{i=1}^3 \brk1{X_{P_i}X_{P_{4m+9-i}}-X_{P_{2m+1+i}} X_{P_{2m+8-i}}}\\
=\ &[e_{54^{m+1}}]X_{P_{4m+9}}
-[e_{4^{(m+1)/2}}]X_{P_{2m+2}} [e_{54^{(m+1)/2}}] X_{P_{2m+7}}
-[e_{54^{(m-1)/2}}]X_{P_{2m+3}} [e_{4^{(m+3)/2}}]X_{P_{2m+6}}\\[4pt]
=\ &3^m(16m+31)
-4\cdot 3^{(m-1)/2}\cdot 3^{(m-1)/2}(8m+23)
-3^{(m-3)/2}(8m+7)\cdot 4\cdot 3^{(m+1)/2}\\[3pt]
=\ &-3^{m-1}(16m+27)<0.
\end{align*}
Now suppose that $m$ is an even positive integer.
The positivity of $S(8,5,3)$ can be verified by direct computation with the aid of~\cref{lem:rec:Sabc}. Let $m\ge 4$. By~\cref{lem:rec:Sabc,prop:Wolfe},
\begin{align*}
&[e_{5^54^{m-4}}]X_{G}\\[-5pt]
=\ &[e_{5^54^{m-4}}]X_{P_{4m+9}}
-[e_{5^54^{m-4}}]\sum_{i=1}^3 X_{P_{2m+1+i}} X_{P_{2m+8-i}}\\
=\ &[e_{5^54^{m-4}}]X_{P_{4m+9}}
-[e_{5^24^{m/2-2}}]X_{P_{2m+2}} [e_{5^34^{m/2-2}}]X_{P_{2m+7}}
-[e_{5^34^{m/2-3}}]X_{P_{2m+3}} [e_{5^24^{m/2-1}}]X_{P_{2m+6}}\\[4pt]
&
-[e_{4^{m/2+1}}]X_{P_{2m+4}} [e_{5^54^{m/2-5}}]X_{P_{2m+5}}
-[e_{5^44^{m/2-4}}]X_{P_{2m+4}} [e_{54^{m/2}}]X_{P_{2m+5}}\\
=\ &\binom{m+1}{5}\cdot4^4\cdot3^{m-5}\brk3{20-\frac{4m+9}{m+1}}\\
&-\binom{m/2}{2}\cdot4\cdot3^{m/2-3}\brk3{20-\frac{2m+2}{m/2}}
\cdot\binom{m/2+1}{3}\cdot4^2\cdot3^{m/2-3}\brk3{20-\frac{2m+7}{m/2+1}}\\
&-\binom{m/2}{3}\cdot4^2\cdot3^{m/2-4}\brk3{20-\frac{2m+3}{m/2}}
\cdot\binom{m/2+1}{2}\cdot4\cdot3^{m/2-2}\brk3{20-\frac{2m+6}{m/2+1}}\\
&-4\cdot3^{m/2}\cdot\binom{m/2}{5}\cdot4^4\cdot3^{m/2-6}\brk3{20-\frac{2m+5}{m/2}}\\
&-\binom{m/2}{4}\cdot4^3\cdot3^{m/2-5}\brk3{20-\frac{2m+5}{m/2}}
\cdot 3^{m/2-1}\brk1{20(m/2+1)-(2m+5)}\\
=\ &-\frac{4}{5}\cdot 3^{m-7}\brk1{32m^5-300m^4+1475m^3-2970m^2+2048m-240}.
\end{align*}
The positivity of $F(m)=32m^5-300m^4+1475m^3-2970m^2+2048m-240$ for even integers $m\ge 4$ can be seen by a direct check for $m\in\{4,6,8\}$ and by
\[
F(m)=(32m-300)m^4+(1475m-2970)m^2+(2048m-240)>0
\]
for $m\ge 10$. Therefore, $[e_{5^54^{m-4}}]X_{G}<0$ for even $m\ge 4$.
This completes the proof.
\end{proof}

Now we concentrate on the positivity of $S(a,b,1)$ for even $b$.

\begin{theorem}\label{thm:spider:ab1}
Let $b$ be even and $a\ge b\ge 2$.
If the spider $S(a,b,1)$ is $e$-positive,
then 
we have the following.
\begin{enumerate}
\item\label[itm]{itm:spider:ab1:b=2mod3}
If $b\equiv2\pmod3$, then
either 
\begin{itemize}
\item
$a\equiv0\pmod3$ and $a\le 2b+2$, or 
\item
$a\equiv1\pmod3$ and $a\le 2b-3$.
\end{itemize}
\item\label[itm]{itm:spider:ab1:b<>2mod3}
If $b\not\equiv2\pmod3$, then $a\le b^2-1$ or $a=b^2+b$.
\item\label[itm]{itm:spider:ab1:ab:even}
If $a$ is even, then
$a>b+(1+\sqrt{8b-3})/2$.
\end{enumerate}
\end{theorem}
\begin{proof}
Let $G=S(a,b,1)$ and $a=(b+1)n+r$, where $r\in\{0,1,\dots,b\}$.
Let $N$ be the number of vertices of $G$. Then
\[
N=(b+1)(n+1)+r+1.
\]
Consider $G$ as obtained by adding a pending edge $v_{b+1}v_N$ to the path $v_1\dotsm v_{N-1}$. By \cref{lem:rec:Sabc},
\begin{equation}\label{rec:spider:ab1}
X_G=e_1 X_{P_{N-1}}+X_{P_N}-X_{P_{a+1}}X_{P_{b+1}}.
\end{equation}
We will extract certain $e$-coefficient from both sides of \cref{rec:spider:ab1}
and show its negativity using \cref{prop:Wolfe}. 

\noindent\cref{itm:spider:ab1:b=2mod3}
Suppose that $b\equiv2\pmod3$. 
If $a\equiv0\pmod3$, then $N\equiv1\pmod3$ and
\begin{align*}
[e_{43^{(N-4)/3}}]X_G
&=[e_{43^{(N-4)/3}}]X_{P_N}
-[e_{43^{(a-3)/3}}]X_{P_{a+1}}[e_{3^{(b+1)/3}}]X_{P_{b+1}}\\
&=2^{(N-1)/3-3}(-3a+6b+7).
\end{align*}
Thus the $e$-positivity of $G$ together with the residue of $a$ implies that $a\le 2b+2$.
If $a\equiv1\pmod3$, then $N\equiv2\pmod3$ and
\begin{align*}
[e_{3^{(N-2)/3}2}]X_G
&=[e_{3^{(N-2)/3}2}]X_{P_N}
-[e_{3^{(a-1)/3}2}]X_{P_{a+1}}[e_{3^{(b+1)/3}}]X_{P_{b+1}}\\
&=2^{(N-2)/3-2}(-a+2b-1).
\end{align*}
Thus the $e$-positivity of $G$ implies $a\le 2b-3$.
If $a\equiv2\pmod3$, then $N\equiv0\pmod3$ and
\[
[e_{3^{N/3}}]X_G
=[e_{3^{N/3}}]X_{P_N}-[e_{3^{(a+1)/3}}]X_{P_{a+1}}[e_{3^{(b+1)/3}}]X_{P_{b+1}}
=-3\cdot2^{N/3-2}<0.
\]

\noindent\cref{itm:spider:ab1:b<>2mod3}
We proceed according to the value of $r$.
When $r=b$, we see that $G$ is not $e$-positive by taking $m=b+1$ in \cref{lem:e:spider:mod}.
If $1\le r\le b-1$, then 
\begin{align*}
[e_{(b+1)^{n+1}(r+1)}]X_G
&=[e_{(b+1)^{n+1}(r+1)}]X_{P_N}-[e_{(b+1)^n(r+1)}]X_{P_{a+1}}[e_{b+1}]X_{P_{b+1}}\\
%&=b^n(brn+rn+2br+b+r)-(b+1)b^{n-1}(brn+rn+br+b)\\
&=b^{n-1}r\brk[s]1{b^2-(b+1)n}-b^n,
\end{align*}
which is negative as if $n\ge b-1$.
If $r=0$, then
\begin{align*}
[e_{(b+2)(b+1)^n}]X_G
&=[e_{(b+2)(b+1)^n}]X_{P_N}-[e_{(b+2)(b+1)^{n-1}}]X_{P_{a+1}}[e_{b+1}]X_{P_{b+1}}\\
%&=b^{n-1}\brk[s]1{(b+1)^2 n+b^2+2b}-(b+1)b^{n-2}\brk[s]1{(b+1)^2 n-1}\\
&=b^{n-2}\brk[s]1{(b+1)^2(b-n)+1},
\end{align*}
which is negative as if $n\ge b+1$.
Summing up the results above yields either $a\le b^2-1$ or $a=b^2+b$.

\noindent\cref{itm:spider:ab1:ab:even}
Now, we suppose that $a$ and $b$ are even.
By \cref{rec:spider:ab1},
\begin{align*}
[e_{3^22^{N/2-3}}]X_G
&=[e_{3^22^{N/2-3}}]X_{P_N}
-[e_{32^{a/2-1}}]X_{P_{a+1}}[e_{32^{b/2-1}}]X_{P_{b+1}}\\
&=a^2-(2b+1)a+(b^2-b+1).
\end{align*}
Thus the $e$-positivity of $G$ implies that $a\ge b+(1+\sqrt{8b-3})/2$,
in which the equality does not hold since $\sqrt{8b-3}$ is not an integer.
This completes the proof.
\end{proof}

%Note that for spiders $S(a,b,1)$, 
%from the proof we see that 
%\[
%a\in
%\{b+1,\,b+2,\,\dots,\,b^2-1\}
%\cup\{b^2+b\}
%\setminus\{k(b+1)-1\colon 2\le k\le b-1\}.
%\]
%This seemingly more restricted range for $a$ is in fact improved from 
%the statement of \cref{thm:spider:ab1} by only removing those we already know 
%that can be removed by \cref{lem:e:spider:mod}.

\Cref{thm:spider:ab1} is sharp in the following sense:
\cref{itm:spider:ab1:b=2mod3}
The spiders $S(6,2,1)$, $S(18,8,1)$ and $S(30,14,1)$ are $e$-positive, as well as the 
the spiders $S(13,8,1)$, $S(25,14,1)$ and $S(37,20,1)$.
\cref{itm:spider:ab1:b<>2mod3}
The spiders $S(15,4,1)$ and $S(35,6,1)$ are $e$-positive. 
\cref{itm:spider:ab1:ab:even}
The spiders $S(6,2,1)$, $S(8,4,1)$ and $S(10,6,1)$ are $e$-positive.

%We also note that the spider $S(14,8,1)$ is not $e$-positive, which implies that 
%the lower bound $b+(1+\sqrt{8b-3})/2$ for even $a$ is improvable.

\begin{corollary}
We have the following.
\begin{enumerate}
\item
The spider $S(a,2,1)$ is $e$-positive if and only if $a\in\{3,6\}$.
\item
The spider $S(a,4,1)$ is $e$-positive if and only if $a\in\{5,8,10,12,13,15,20\}$.
\item
The spider $S(a,6,1)$ is $e$-positive if and only if 
$a\in\{7,8,\dots,35\}\cup\{42\}\setminus\{8,13,20,27,34\}$.
\item
The spider $S(a,8,1)$ is $e$-positive if and only if 
$a\in\{9,13,15,18\}$.
\end{enumerate}
\end{corollary}
\begin{proof}
By \cref{thm:spider:ab1}, it suffices to check the $e$-positivity 
of a few number of spiders for each spider classes $S(a,b,1)$ with $b\in\{2,4,6,8\}$.
One may compute the chromatic symmetric function of these spiders straightforwardly by using \cref{rec:spider:ab1,prop:gf:path:cycle}.
\end{proof}

In view of the sporadic case $a=b^2+b$ in \cref{itm:spider:ab1:b<>2mod3},
we propose \cref{conj:ab1:b=even:b<>2},
which is checked to be true for $b\in\{4,6\}$.

\begin{conjecture}\label{conj:ab1:b=even:b<>2}
If $b$ is even and $b\not\equiv2\pmod3$,
then the spider $S(b^2+b,\,b,\,1)$ is $e$-positive.
\end{conjecture}

We further conjecture that all spiders $S(a,b,1)$
that have been shown not to be $e$-positive in \cref{thm:spider:ab1}
are Schur positive.

\begin{conjecture}\label{conj:spider:ab1}
Suppose that $a\ge b\ge 2$ and $b$ is even.
The spider $S(a,b,1)$ is Schur positive if one of the following is true.
\begin{enumerate}
%\item\label[itm]{itm:conj:ab1:b=2}
%$b\equiv2\pmod3$.
\item\label[itm]{itm:conj:ab1:b=2:a=0}
$b\equiv2\pmod3$, $a\equiv0\pmod3$ and $a\ge 2b+5$.
\item\label[itm]{itm:conj:ab1:b=2:a=1}
$b\equiv2\pmod3$, $a\equiv1\pmod3$ and $a\ge 2b$.
\item\label[itm]{itm:conj:ab1:b=2:a=2}
$b\equiv2\pmod3$, $a\equiv2\pmod3$ and $a\ge b$.
\item\label[itm]{itm:conj:ab1:b<>2}
$b\not\equiv2\pmod3$ and $a\ge b^2$.
\item\label[itm]{itm:conj:ab1:a=even}
$a$ is even and $b\le a\le b+(1+\sqrt{8b-3})/2$.
\end{enumerate}
\end{conjecture}

It is routine to check \cref{conj:spider:ab1} for the first few values of the pair $(a,b)$: \cref{itm:conj:ab1:b=2:a=0} is true for $b=2$ and $a\le 30$, as well as for $b=8$ and $a\le 24$;
\cref{itm:conj:ab1:b=2:a=1} is true for $b=2$ and $a\le 28$, as well as for $b=8$ and $a\le 22$; 
\cref{itm:conj:ab1:b=2:a=2} is true for $b=2$ and $a\le 29$, as well as for $b=8$ and $a\le 20$, and $b=a=14$; 
\cref{itm:conj:ab1:b<>2} is true for $b=4$ and $a\le 25$; 
\cref{itm:conj:ab1:a=even} is true for $b\le 12$. 

\begin{conjecture}\label{conj:spider:ab1:new}
Suppose that $a\ge b\ge 2$ and $b$ is even.
The spider $S(a,b,1)$ is Schur positive if one of the following is true. 
\begin{enumerate}
\item\label[itm]{itm:conj:ab1:new:b=2:a=0}
$b\equiv2\pmod3$, $a\equiv0\pmod3$ and $a\le 2b+2$.
\item\label[itm]{itm:conj:ab1:new:b=2:a=1}
$b\equiv2\pmod3$, $a\equiv1\pmod3$ and $a\le 2b-3$.
\item\label[itm]{itm:conj:ab1:new:b<>2}
$b\not\equiv2\pmod3$ and $b\le a\le b^2-1$.
\item\label[itm]{itm:conj:ab1:new:a=even}
$a$ is even and $a>b+(1+\sqrt{8b-3})/2$.
\end{enumerate}
\end{conjecture}
It is routine to check \cref{conj:spider:ab1:new} for the first few values of the pair $(a,b)$: \cref{itm:conj:ab1:new:b=2:a=0,itm:conj:ab1:new:b=2:a=1} are true for $b\in\{2,8\}$;
\cref{itm:conj:ab1:new:b<>2} is true for $b=4$, and for $b=6$ with $a\le 23$; 
\cref{itm:conj:ab1:new:a=even} is true for ***.

\section{The positivity of spiders $S(a,b,2)$}\label{sec:spider:ab2}

This section is devoted to the $e$-positivity and Schur positivity of spiders $S(a,b,2)$.

\begin{theorem}\label{thm:e:spider:ab2}
Let $a\ge b\ge 2$. Suppose that the spider $S(a,b,2)$ is $e$-positive.
Let $r_a$ and $r_b$ be the remainders of $a$ and $b$ modulo 3, respectively.
Then either $(r_a,r_b)=(0,1)$ or $r_b=0$.
\end{theorem}
\begin{proof}
Write $G=S(a,b,2)$. 
By \cref{lem:e:spider:mod}, we find $r_a+r_b\le 2$.
Before proceeding according to the values of $r_a$ and $r_b$,
we need a recurrence to compute $X_G$. 
The order of the graph $G$, denoted $N$, is $a+b+3$.
For any partition $\lambda\vdash N$ such that every part of $\lambda$ is at least $3$, we can extract the coefficient of $e_\lambda$ by \cref{lem:rec:Sabc} and obtain
\begin{equation}\label{pf:rec:spider:ab2}
[e_\lambda]X_G
=[e_\lambda]X_{P_N}-[e_\lambda]X_{P_{a+1}}X_{P_{b+2}}-[e_\lambda]X_{P_{a+2}}X_{P_{b+1}}.
\end{equation}

If $r_a=r_b=1$, then $a,b\ge 4$ and $N=a+b+3\ge 11$. 
Setting $\lambda=53^{(N-5)/3}$ in \cref{pf:rec:spider:ab2} and using \cref{prop:Wolfe} we can deduce that 
\begin{align*}
&[e_{53^{(N-5)/3}}]X_G\\
=\ &[e_{53^{(N-5)/3}}]X_{P_N}
-[e_{53^{(a-4)/3}}]X_{P_{a+1}}[e_{3^{(b+2)/3}}]X_{P_{b+2}}
-[e_{3^{(a+2)/3}}]X_{P_{a+2}}[e_{53^{(a-4)/3}}]X_{P_{b+1}}\\
=\ &2^{(N-5)/3}(13-N).
\end{align*}
It is negative unless $N\le 13$. Suppose that $N\le 13$.
Since $a\ge b\ge 4$ and $N=a+b+3$,
we find $a=b=4$. In this case, the graph $G$ is $S(4,4,2)$.
It is not $e$-positive by \cref{thm:lambda1:d:n}.

If $(r_a,r_b)=(0,2)$, then $a\ge b+4$ by \cref{thm:lambda1:d:n},
and $N=a+b+3\ge 11$. 
Setting $\lambda=4^23^{(N-8)/3}$ in \cref{pf:rec:spider:ab2} and using \cref{prop:Wolfe} we can deduce that 
\begin{align}
&[e_{4^23^{(N-8)/3}}]X_G\notag\\
=\ &[e_{4^23^{(N-8)/3}}]X_{P_N}
-[e_{43^{(a-3)/3}}]X_{P_{a+1}}[e_{43^{(b-2)/3}}]X_{P_{b+2}}
-[e_{4^23^{(a-6)/3}}]X_{P_{a+2}}[e_{3^{(b+1)/3}}]X_{P_{b+1}}\notag\\
=\ &2^{(N-17)/3}\brk1{-3a^2+(11-6b)a+6b^2-4b-18}\label{pf:e:spider:a0b2}\\
=\ &2^{(N-17)/3}\brk1{-3b^2-12tb-41b-3t^2-13t-22},\notag
\end{align}
where $t=a-b-4\ge 0$. It is negative.
This completes the proof.
\end{proof}

For the remaining possible $e$-positive spiders $S(a,b,2)$,
we first give an upper bound of $a$ in terms of $b$ in \cref{thm:ab2:ub},
for which we need \cref{lem:k:k+1}.

\begin{lemma}\label{lem:k:k+1}
Let $k\ge 2$.
The set of positive integers $n$ for which
the equation $n=xk+y(k+1)$ has a solution $(x,y)\in\mathbb{N}^2$ is
\[
\{qk+r\colon 1\le q\le k-2,\,0\le r\le q\}\cup\{n\colon n\ge k(k-1)\}.
\]
\end{lemma}
\begin{proof}
The desired set is 
\[
\{xk+y(k+1)\colon x,y\in\mathbb N\}
=\{(x+y)k+y\colon x,y\in\mathbb N\}
=\{qk+r\colon q\ge r\ge 0,\,q\ge 1\}.
\]
Suppose that $n=qk+r$, where $q\in\mathbb{N}$ and $0\le r\le k-1$.
If $n\ge k(k-1)$, then $(x,y)=(q-r,\,r)$ is a solution in $\mathbb N^2$ 
since $q\ge k-1\ge r$.
Otherwise $n<k(k-1)$ and 
\[
n\in\{qk+r\colon 1\le q\le k-2,\,0\le r\le q\}.
\]
This completes the proof.
\end{proof}

\begin{theorem}\label{thm:ab2:ub}
If the spider $S(a,b,2)$ is $e$-positive,
then 
\[
a\in\bigcup_{q=1}^{b-2}\{m\in\mathbb N\colon (b+2)q\le m\le (b+1)q+b-2\}.
\]
Moreover, if $(r_a,r_b)=(0,1)$, 
where $r_a$ and $r_b$ are the remainders of $a$ and $b$ modulo 3 respectively,
then $a\le 2b+4$.
\end{theorem}
\begin{proof}
Let $G=S(a,b,2)$.
Then the number $N$ of vertices in $G$ is $a+b+3$.
Suppose that $N$ has a partition $(b+2)^y(b+1)^x$ for some integers $x,y\in\mathbb N$.
By easy combinatorial arguments, 
the spider $S(a,b,2)$ does not contain a connected partition whose blocks are of sizes $b+1$ or $b+2$, contradicting \cref{thm:Wolfgang}.
Therefore, the equation $N=x(b+1)+y(b+2)$ has no solution $(x,y)\in\mathbb N^2$.
Taking the complement of the set in \cref{lem:k:k+1}, we obtain the desired range of $a$.

If $(r_a,r_b)=(0,1)$, then $b\ge 4$, $a\ge 6$ and $N\ge 13$.
Setting $\lambda=43^{(N-4)/3}$ in \cref{pf:rec:spider:ab2} and using \cref{prop:Wolfe} we can deduce that 
\begin{align*}
[e_{43^{(N-4)/3}}]X_G
&=[e_{43^{(N-4)/3}}]X_{P_N}
-[e_{43^{(a-3)/3}}]X_{P_{a+1}}[e_{3^{(b+2)/3}}]X_{P_{b+2}}\\
&=2^{(N-10)/3}(-3a+6b+13).
\end{align*}
Since $[e_{43^{(N-4)/3}}]X_G\ge 0$, we find $a\le 2b+4$.
\end{proof}

\begin{corollary}
The spiders $S(a,b,2)$ with $b\le11$ satisfy the following.
\begin{enumerate}
\item
No spider $S(a,b,2)$ with $b\equiv2\pmod3$ is $e$-positive.
\item
When $b\in\{3,6,7\}$,
the spider $S(a,b,2)$ is $e$-positive if and only if $a=b+2$.
\item
The spider $S(a,4,2)$ is $e$-positive if and only if $a\in\{6,12\}$.
\item
No spider $S(a,9,2)$ or $S(a,10,2)$ is $e$-positive.
\end{enumerate}
\end{corollary}
\begin{proof}
By direct computation with the aid of \cref{thm:e:spider:ab2,thm:ab2:ub}.
We take $S(a,9,2)$ with $a\in\{56,57,66,67,77\}$ for example.
When $a\equiv 1\pmod5$ and $a\ge 11$,
\begin{align*}
[e_{6^35^{(a-6)/5}}]X_{S(a,9,2)}
&=[e_{6^35^{(a-6)/5}}]X_{P_{a+12}}
-[e_{5^2}]X_{P_{10}}[e_{6^35^{(a-16)/5}}]X_{P_{a+2}}
-[e_{65}]X_{P_{11}}[e_{6^25^{(a-11)/5}}]X_{P_{a+1}}\\
&=-\frac{2^{(2a-37)/5}}{3}\brk1{5a^3+42a^2-3533a+4206},
\end{align*}
which is negative for $a\ge 26$. In particular,
$[e_{6^35^{10}}]X_{S(56,9,2)}<0$ and $[e_{6^35^{12}}]X_{S(66,9,2)}<0$.
When $a\equiv 2\pmod5$ and $a\ge 12$,
\[
[e_{5^{(a+8)/5}4}]X_{S(a,9,2)}
=[e_{5^{(a+8)/5}4}]X_{P_{a+12}}
-[e_{5^2}]X_{P_{10}}[e_{5^{(a-2)/5}4}]X_{P_{a+2}}
=2^{(2a-4)/5}(110-3a),
\]
which is negative for $a\ge 37$. In particular,
$[e_{5^{13}4}]X_{S(57,9,2)}$, $[e_{5^{15}4}]X_{S(67,9,2)}$
and $[e_{5^{17}4}]X_{S(77,\,9,\,2)}$ are negative.
This completes the proof.
\end{proof}

When $(r_a,r_b)=(0,1)$, the upper bound $2b+4$ of $a$ is sharp 
in the sense that the spider $S(10,4,2)$ is $e$-positive.
Now we give a lower bound of $a$ for $b$ that is divisible by 3
for $e$-positive spiders $S(a,b,2)$.

\begin{theorem}\label{thm:e:spider:ab22}
Let $a\ge b\ge 12$. Suppose that $3$ divides $b$ and the spider $S(a,b,2)$ is $e$-positive. Then 
\[
a\ge \begin{cases}
3b+3,&\text{if $a\equiv0\pmod3$},\\[5pt]
\displaystyle \frac{b}{2}+\frac{9}{4}+\frac{\sqrt{12b^2-180b+565}}{4},
&\text{if $a\equiv1\pmod3$},\\[9pt]
\displaystyle \frac{b}{2}+\frac{1}{3}+\frac{\sqrt{27b^2-54b+112}}{6},
&\text{if $a\equiv2\pmod3$}.
\end{cases}
\]
\end{theorem}
\begin{proof}
Let $G=S(a,b,2)$. Denote by $N$ the number of vertices in $G$. Then $N=a+b+3$. We proceed according to the remainder of $a$ modulo 3.

Suppose that $a\equiv0\pmod3$. By \cref{thm:lambda1:d:n}, 
we find $a\ge b+3\ge 15$. Thus $N\ge 30$.
Setting $\lambda=4^33^{(N-12)/3}$ in \cref{pf:rec:spider:ab2} and using \cref{prop:Wolfe} we can deduce that 
\begin{align*}
&[e_{4^33^{(N-12)/3}}]X_G\\
=\ &[e_{4^33^{(N-12)/3}}]X_{P_N}
-[e_{43^{(a-3)/3}}]X_{P_{a+1}}[e_{4^23^{(b-6)/3}}]X_{P_{b+2}}-
[e_{4^23^{(a-6)/3}}]X_{P_{a+2}}[e_{43^{(b-3)/3}}]X_{P_{b+1}}\\
=\ &2^{(N-21)/3}\brk1{(9N-87)b^2-(9N-87)(N-3)b+2N^3-35N^2+181N-306}.
\end{align*}
Consider the function 
\[
F(b)=(9N-87)b^2-(9N-87)(N-3)b+2N^3-35N^2+181N-306.
\]
Since 
$N=a+b+3\ge 2b+6$ by \cref{thm:lambda1:d:n},
we find $b\le (N-6)/2$. 
Note that $F(b)$ is decreasing for $b\le (N-3)/2$.
If $a\le 2b+3$, then $N=a+b+3\le 3b+6$, that is,
$b\ge (N-6)/3$. Since
\[
F\brk4{\frac{N-6}{3}}=-\frac{2}{3}N^2+18N-132<0,
\]
we deduce that $F(b)<0$, contradicting to the fact $[e_{4^33^{(N-12)/3}}]X_G\ge 0$. This proves that $a\ge 2b+6$. 

Below we suppose that $2b+6\le a\le 3b$. Then $a\ge 28$ and $N\ge 42$.
Setting $\lambda=4^63^{(N-24)/3}$ in \cref{pf:rec:spider:ab2} and using \cref{prop:Wolfe} we can deduce that 
\begin{align*}
&[e_{4^63^{(N-24)/3}}]X_G\\
=\ &[e_{4^63^{(N-24)/3}}]X_{P_N}
-[e_{43^{(a-3)/3}}]X_{P_{a+1}}[e_{4^53^{(b-18)/3}}]X_{P_{b+2}}
-[e_{4^43^{(a-15)/3}}]X_{P_{a+1}}[e_{4^23^{(b-6)/3}}]X_{P_{b+2}}\\
& -[e_{4^23^{(a-6)/3}}]X_{P_{a+2}}[e_{4^43^{(b-15)/3}}]X_{P_{b+1}}-
[e_{4^53^{(a-18)/3}}]X_{P_{a+2}}[e_{43^{(b-3)/3}}]X_{P_{b+1}}\\
=\ &\frac{2^{N/3-14}}{15} H(a),
\end{align*}
where
\begin{align*}
H(a)&=
2{a}^{6}-(6b+124){a}^{5}
+(-15{b}^{2}+355b+3020){a}^{4}
+(40{b}^{3}+110{b}^{2}-6640b-37440){a}^{3}\\
&\qquad
+(-15{b}^{4}+110{b}^{3}-11250{b}^{2}+125415b+124938){a}^{2}\\
&\qquad
+(-6{b}^{5}+355{b}^{4}-6640{b}^{3}+125415{b}^{2}-945468b+408564)a\\
&\qquad
+(2{b}^{6}-124{b}^{5}+3020{b}^{4}-37440{b}^{3}+124938{b}^{2}+408564b+699840).
\end{align*}
First, it is routine to check that for $b\ge 3$,
\[
H(2b+6)
=-54b^6+2646b^5-43590b^4+184770b^3+194724b^2-2243376b+2604960<0.
\]
In order to show that $H(a)<0$ in the interval $[2b+6,\,3b]$,
it suffices to show that the differential~$H'(a)$ is negative for $a\in[2b+6,\,3b]$. In fact,
\begin{align*}
H'(a)
&=12a^5(30b+620)a^4
+(-60b^2+1420b+12080)a^3
+(120b^3+330b^2-19920b-112320)a^2\\
&\qquad
+(-30b^4+220b^3-22500b^2+250830b+249876)a\\
&\qquad
+(-6b^5+355b^4-6640b^3+125415b^2-945468b+408564).
\end{align*}
It is routine to check that for $b\ge3$,
\begin{align*}
H'(2b+6)&=-162b^5+1935b^4-42240b^3+256275b^2-392256b-236628<0
\quad\text{and}\\
H'(3b)&=-150b^5-7895b^4+72740b^3-132975b^2-195840b+408564<0.
\end{align*}
Therefore, it suffices to show that
the second differential $H''(a)$ satisfies the following two properties:
\begin{itemize}
\topsep=0pt
\item
$H''(2b+6)<0$ for $b\ge 3$, and
\item
$H''(a)$ has at most one real root in the inteval $[2b,\,3b]$.
\end{itemize}
In fact,
\begin{align*}
H''(a)
&=60a^4-(120b+2480)a^3
+(-180b^2+4260b+36240)a^2\\
&\qquad+(240b^3+660b^2-39840b-224640)a
+(-30b^4+220b^3-22500b^2+250830b+249876).
\end{align*}
It is routine to check that for $b\ge 3$,
\[
H''(2b+6)
=-270b^4-1260b^3-10140b^2+127710b-251244<0.
\]
In order to prove that $H''(a)$ has at most one real root in $[2b,\,3b]$,
it suffices to show that $H''(a)$ has three distinct real roos less than $2b$.
In fact, it is routine to check that for $b\ge 12$,
\begin{align*}
H''(0)&=-30b^4+220b^3-22500b^2+250830b+249876<0,\\
H''\brk3{\frac{b}{2}}&=\frac{135}{4}b^4+1305b^3-33360b^2+138510b+249876>0.
\end{align*}
Note that the function $H''(a)$ is a polynomial of degree 4 with positive leading coefficient. By the intermediate value theorem, we derive that $H''(a)$ has a real root in the intervals
\[
(-\infty,0),\quad
(0,b/2),\quad\text{and}\quad
(b/2,2b),
\]
respectively. This completes the proof for the fact $H(a)<0$, 
contradicting $[e_{4^63^{(N-24)/3}}]X_G\ge 0$. Hence $a\ge 3b+3$.

If $a\equiv1\pmod3$, then $a\ge 16$ and $N\ge 31$.
Setting $\lambda=5^23^{(N-10)/3}$ in \cref{pf:rec:spider:ab2} and using \cref{prop:Wolfe} we can deduce that 
\begin{align*}
&[e_{5^23^{(N-10)/3}}]X_G\\
=\ &[e_{5^23^{(N-10)/3}}]X_{P_N}
-[e_{53^{(a-4)/3}}]X_{P_{a+1}}[e_{53^{(b-3)/3}}]X_{P_{b+2}}-
[e_{3^{(a+2)/3}}]X_{P_{a+2}}[e_{5^23^{(b-9)/3}}]X_{P_{b+1}}\\
=\ &\frac{2^{(N-10)/3}}{3}\brk1{4a^2-(4b+18)a-2b^2+54b-121}.
\end{align*}
The $e$-positivity of $G$ implies the desired inequality.

If $a\equiv2\pmod3$, then $a\ge 14$ and $N\ge 29$. 
Exchanging the letters $a$ and $b$ in \cref{pf:e:spider:a0b2}, and sorting the terms according to the degree of $a$, we obtain
\begin{align*}
[e_{4^23^{(N-8)/3}}]X_G
&=2^{(N-17)/3}\brk1{-3b^2+(11-6a)b+6a^2-4a-18}\\
&=2^{(N-17)/3}\brk1{6a^2-(6b+4)a-3b^2+11b-18}.
\end{align*}
The $e$-positivity of $G$ implies the desired inequality.
\end{proof}

Note that in the formula in \cref{thm:e:spider:ab22},
the two discriminants $12b^2-180b+565$ and $27b^2-54b+112$
are positive for $b\ge 12$.
%Note that the spider $S(a,3,2)$ is not $e$-positive for $a\ge 6$,
%since it does not contain a connnected partition whose blocks are of sizes 5 or 4. From this point of view, the three results for $r_b=0$ in \cref{thm:e:spider:ab22} are not sharp. 
%

\section{The positivity of broom graphs and double broom graphs}\label{sec:broom}

We call the spider $S(\lambda_1,1^{d-1})$ a \emph{broom}, denoted 
\[
S(\lambda_1,1^{d-1})=\mathrm{br}(\lambda_1,d-1).
\]
When $\lambda_1\ge 2$, we call the path $P_{1+\lambda_1}$
the \emph{long leg} of the broom $\mathrm{br}(\lambda_1,d-1)$.
We have the following complete positivity classification for the family of brooms.

\begin{theorem}
The positivity classification of brooms $\{\mathrm{br}(p,l)\colon p,l\ge 2\}$ is as follows.
\begin{enumerate}
\item
$\mathrm{br}(p,l)$ is $e$-positive if and only if $p=l=2$.
\item
$\mathrm{br}(p,l)$ is Schur positive but not $e$-positive if and only if $p\in\{4,6,8,10,12\}$ and $l=2$.
\item
$\mathrm{br}(p,l)$ is not Schur positive if $p\not\in\{2,4,6,8,10,12\}$ or $l\ge 3$.
\end{enumerate}
\end{theorem}
\begin{proof}
Since $\mathrm{br}(p,l)$ is a tree, it is bipartite.
When $l\ge 3$ or $p$ is odd, the broom $\mathrm{br}(p,l)$ is not balanced,
and thus not Schur positive by \cref{thm:nS:balanced}.
Suppose that $l=2$ and $p$ is even.
By using mathematical software,
one may compute $X_{\mathrm{br}(p,l)}$ 
and obtain that $\mathrm{br}(2,2)$ is $e$-positive, 
and that $\mathrm{br}(p,2)$ for $p\in\{4,6,8,10,12\}$ is Schur positive.
Moreover,
by using \cref{lem:2odds}, we find $[e_{32^*}]X_G=5-2p<0$ for even $p\ge 4$.

Below we deal with the remaining brooms $G=\mathrm{br}(2p,2)$ with $2p\ge 14$.
We shall show that $[s_\lambda]X_G=6-p$ by using \cref{thm:s-in-X}, where
$\lambda=(p+1)^21$.
We label the center of $G$ as $v_{2p+1}$, the long leg as $v_{2p+1}v_{2p}\dotsm v_1$,
and the remaining two vertices as $u$ and $w$, see \cref{fig:broom:2p:2}.

\begin{figure}[htbp]
\begin{center}
\begin{tikzpicture}[scale=0.7]
\input{broom}
\end{tikzpicture}
\caption{The broom $\mathrm{br}(2p,2)$.}\label{fig:broom:2p:2}
\end{center}
\end{figure}
First of all, we note that the independence number $\alpha$ of $G$ is $\alpha=p+2$,
since the path $v_1\dotsm v_{2p}$ has independence number $p$ 
and the path $uv_{2p+1}w$ has independence number~$2$.
In view of \cref{thm:s-in-X}, only 3 rim hook tabloids of shape $\lambda$ need consideration, which are illustrated in \cref{fig:RHT:broom:2p:2} 
with their contents $\tau$ respectively.
Now we compute their contributions to the coefficient $[s_\lambda]X_G$ independently.

\begin{figure}[htbp]
\begin{tikzpicture}[scale=0.5]
\input{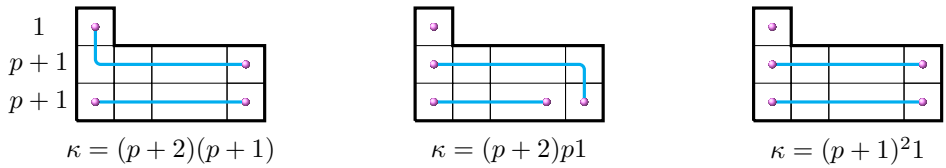}
\end{tikzpicture}
\caption{The special rim hook tabloids in the set $\mathcal{T}_{(p+1)^2 1}$ for the broom~$\mathrm{br}(2p,2)$.}
\label{fig:RHT:broom:2p:2}
\end{figure}

\begin{enumerate}
\item
The graph $G$ has only one stable partition of type $(p+2)(p+1)$, in 
which the stable set of order~$p+2$ is
$\{v_2,v_4,\dots,v_{2p},u,w\}$.
This partition contributes $-1$ to $[s_\lambda]X_G$.
\item
Let $(A,B,C)$ be a stable partition of type $(p+2)p1$, where $\abs{A}=p+2$, $\abs{B}=p$ and $\abs{C}=1$. Then the set $A$ must be of the form
\[
T_j=\{v_1,\,v_3,\,\dots,\,v_{2j-1},\,v_{2j+2},\,v_{2j+4},\,\dots,\,v_{2p},\,u,\,w\}
\quad\text{where $j\in\{0,1,\dots,p\}$.}
\]
If $j=0$, then $C$ consists of an arbitrary vertex in $V(G)\backslash T_0$.
If $1\le j\le p$, then $C$ consists of $v_{2j}$ or $v_{2j+1}$. 
Therefore, the contribution of such partitions is $-(p+1)-2p=-(3p+1)$.
\item
Let $(A,B,C)$ be a stable partition of type $(p+1)^21$, where $\abs{A}=\abs{B}=p+1$ and $\abs{C}=1$.
\begin{enumerate}
\item
If $v_{2p+1}\in A$, then $A=\{v_1,v_3,\dots,v_{2p+1}\}$. Since $V(G)\backslash A$ is stable, there are $p+2$ possibilities for the partition $(B,C)$.
\item
If $v_{2p+1}\in C$, then $V(G)\backslash\{v_{2p+1}\}$ is partitioned into two stable sets.
If $v_1\in A$, then 
\[
\{v_1,\,v_3,\,\dots,\,v_{2p-1}\}\subset A
\quad\text{and}\quad
\{v_2,\,v_4,\,\dots,\,v_{2p}\}\subset B.
\]
Thus one of the vertices $u$ and $w$ belongs to $A$ and the other to $B$.
\end{enumerate}
In summary, the contribution of this kind of partitions is 
$2(p+2+2)=2p+8$, where the factor~$2$ comes from the ordering of $A$ and $B$. 
\end{enumerate}
Hence $[s_\lambda]X_G=-1-(3p+1)+(2p+8)=6-p$, which is negative for $p\ge 7$.
\end{proof}

Denote by $\mathrm{br}'(l,p,l')$ the $(l+p+l'+1)$-vertex graph obtained by
identifing the center of the star~$S(1^{l})$ and 
the leave of the broom $\mathrm{br}(p,l')$ on its long leg.
We have the following complete positivity classification for the family of $\mathrm{br}'(l,p,l')$.

\begin{theorem}
In the graph family
\[
\mathcal{G}=\{\mathrm{br}'(l,p,l')\colon l'\ge l\ge 2,\,l'\ge 3,\,p\ge 1\},
\]
no one is $e$-positive. Moreover, only the following 10 graphs in $\mathcal{G}$ are Schur positive:
\begin{align*}
&\mathrm{br}'(2,1,3),\quad
\mathrm{br}'(2,5,3),\quad
\mathrm{br}'(2,7,3),\quad
\mathrm{br}'(2,9,3),\quad
\mathrm{br}'(2,11,3),\\
&\mathrm{br}'(3,1,3),\quad
\mathrm{br}'(3,1,4),\quad 
\mathrm{br}'(4,1,4),\quad 
\mathrm{br}'(4,1,5),\quad
\mathrm{br}'(5,1,5).
\end{align*}
\end{theorem}
\begin{proof}
Suppose that $\mathrm{br}'(l,p,l')$ is Schur positive.
Since it is bipartite, it is balanced by \cref{thm:nS:balanced}.
It follows that~$p$ is odd and $l'\in\{l,\,l+1\}$. 
For avoiding fractions, we consider the graph
\[
G=\mathrm{br}'(l,\,2p-1,\,l')
\]
for $p\ge 1$. 
We regard $G$ as consisting of the path $v_1v_2\dotsm v_{2p}$,
and the stars with edge sets $\{v_1x\colon x\in X\}$ and $\{v_{2p}y\colon y\in Y\}$,
where $X$ and $Y$ are disjoint vertex sets of orders $l$ and $l'$ respectively, 
see \cref{fig:SPS}.
Let $W=\{v_1,\dots,v_{2p}\}$.
Then $V(G)=X\sqcup W\sqcup Y$ and $\abs{V(G)}=l+l'+2p$.
\begin{figure}[htbp]
\begin{center}
\begin{tikzpicture}[scale=0.6]
\input{SPS}
\end{tikzpicture}
\caption{The graph $\mathrm{br}'(l,\,2p-1,\,l')$.}\label{fig:SPS}
\end{center}
\end{figure}

Our first goal is to show that $l\le 5$. Suppose to the contrary that $l\ge 6$ and consider the partition
\[
\lambda=(l'+p+1)(l+p-2)1.
\]
Let $T$ be a special rim hook tabloid of shape $\lambda$ and some content $\tau$. 
By definition, we know that 
\begin{equation}\label[ineq]{pf:ub:ltau}
\ell(\tau)\le \ell(\lambda)=3,
\end{equation}
and the maximum part $\tau_1$ in $\tau$ satisfies
\begin{equation}\label[ineq]{pf:lb:tau1}
\tau_1\ge \lambda_1=l'+p+1.
\end{equation}
Let $A$ be a stable set in a stable partition of type $\tau$ such that $\abs{A}=\tau_1$. The two inequalties above imply the following further results.
\begin{itemize}
\item
$\{v_1,v_{2p}\}\cap A=\emptyset$. Otherwise,
one of the sets $X$ and $Y$ would have empty intersection with $A$, which implies that $\abs{A}\le l'+p$, contradicting \cref{pf:lb:tau1}.
\item
$\ell(\tau)=3$. Otherwise, one would have $\ell(\tau)=2$ by \cref{pf:ub:ltau}. It follows that the vertices on the path $v_1\dotsm v_{2p}$ are partitioned into two stable sets.
Therefore, the set $A$ must contain exactly one of the endpoints $v_1$ and $v_{2p}$, contradicting the previous result.
\end{itemize}
Thus a feasible type $\tau$ to form a special rim hook tabloid of shape $\lambda$ must be
$(l'+p+2)(l+p-3)1$ or $(l'+p+1)(l+p-2)1$,
see \cref{fig:RHT-SPS}.
\begin{figure}[htbp]
\begin{center}
\begin{tikzpicture}[scale=0.6]
\input{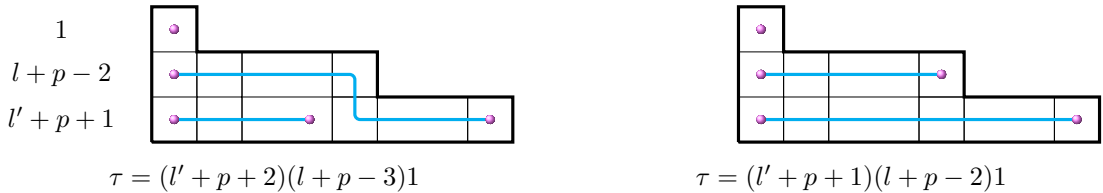}
\end{tikzpicture}
\caption{The special rim hook tabloids in the set $\mathcal{T}_{(l'+p+1)(l+p-2)1}$ for the graph~$\mathrm{br}'(l,\,2p-1,\,l')$.}\label{fig:RHT-SPS}
\end{center}
\end{figure}
In view of \cref{thm:s-in-X},
we need to consider stable partitions $(A,B,C)$ such that 
\[
\abs{B}\in\{l+p-3,\,l+p-2\}\quad\text{and}\quad
\abs{C}=1.
\]
It follows that $\abs{A}>\abs{B}>\abs{C}$ and $\abs{B}\ge p+3$.
If $\{v_1,v_{2p}\}\subset B$,
then the stability of $B$ implies that $B\subset W$. Thus $\abs{B}\le p$, 
a contradiction.
Hence one of the vertices $v_1$ and $v_{2p}$ is in~$B$,
and the other forms the singleton $C$.

Suppose that $v_1\in B$ and $C=\{v_{2p}\}$.
Since each of the vertices $v_2$, $\dots$, $v_{2p-1}$ is in $A$ or $B$,
\[
\{v_1,v_3,\dots,v_{2p-1}\}\subset B
\quad\text{and}\quad
B\backslash\{v_1,v_3,\dots,v_{2p-1}\}\subset Y.
\]
Thus there are $\binom{l'}{\abs{B}-p}$ possibilities 
for the partition $(A,B,C)$.
For the other case that $v_{2p}\in B$ and $C=\{v_1\}$, 
one may derive by symmetry that there are $\binom{l}{\abs{B}-p}$ possibilities for the partition. Note that the inequalities $\abs{A}>\abs{B}>\abs{C}$ guarantee that the ordering of stable sets of the same order can be ignored.
Hence
\[
[s_\lambda]X_G
=\brk4{\binom{l}{l-2}-\binom{l}{l-3}}+\brk4{\binom{l'}{l-2}-\binom{l'}{l-3}},
\]
which is negative for $l\ge 6$. This contradiction implies that $l\le 5$.

The second goal of ours is to show that $p\le l$ if $l\ge 3$.
Suppose to the contrary that $p\ge l+1$ and $l\ge 3$. Consider the partition
\[
\lambda=(l+l'+p-2)(p+1)1.
\]
Let $T$ be a special rim hook tabloid of shape $\lambda$ and content $\tau$. Since $l\ge 3$, we find
\[
\tau_1\ge \lambda_1=l+l'+p-2\ge l'+p+1.
\]
Let $A$ be a stable set in a stable partition of type $\tau$ such that $\abs{A}=\tau_1$. Along the same lines for the first goal, we can derive that 
$\{v_1,v_{2p}\}\cap A=\emptyset$
and $\ell(\tau)=3$. It follows that $\tau$ must be
\[
(l+l'+p-1)p1
\quad\text{or}\quad
(l+l'+p-2)(p+1)1.
\]
Consider stable partitions $(A,B,C)$ such that 
\[
\abs{A}\in\{l+l'+p-1,\,l+l'+p-2\}\quad\text{and}\quad
\abs{C}=1.
\]
It follows that $\abs{A}>\abs{B}>\abs{C}$.

If $\abs{A}=l+l'+p-1$, then 
\[
A=X\cup Y\cup \{v_2,\,v_4,\,\dots,\,v_{2j},\,v_{2j+3},\,v_{2j+5},\,\dots,\,v_{2p-1}\}\quad\text{for some $j\in\{0,1,\dots,p-1\}$}.
\]
For each $j\in\{0,1,\dots,p-1\}$, the singleton $C$ must be $\{v_{2j+1}\}$ or $\{v_{2j+2}\}$. Thus the negative part in $[s_\lambda]X_G$ by \cref{thm:s-in-X}
is $-2p$.

Suppose that $\abs{A}=l+l'+p-2$.
Since $\abs{A\cap W}\le p-1$ and $\abs{X\cup Y}=l+l'$, 
we find $\abs{A\cap W}\in\{p-1,\,p-2\}$. Denote $W'=W\backslash A$.
Since $l\ge 2$, we find $\{v_1,v_{2p}\}\subset W'$.
Suppose that $\abs{A\cap W}=p-2$. 
Then $X\cup Y\subset A$.
The number of components of the induced subgraph $G[W']$ is 
$\abs{A\cap W}+1$.
Since the graph $G[W']$ is partitioned into the stable set $B$ and the singleton~$C$,
at most one of its components is not an isolated vertex, 
and that component (it it exists) must be the path $P_2$ or $P_3$.
Since $\abs{A\cap W}=p-2$, we find $\abs{W'}\le p+1$.
It follows that $\abs{W}\le 2p-1$, which is absurd.
This proves 
\[
\abs{A\cap W}=p-1.
\]
As a consequence, the set $X\cup Y\backslash A$ is a singleton, say, $\{u\}$.
Again, the induced subgraph $G[W'\cup\{u\}]$ is partitioned into the stable set $B$ and the singleton $C$, and consists of several isolated vertices and the path $P_2$ or $P_3$. 
\begin{itemize}
\item
If $u\in X$, then $G[W'\cup\{u\}]$ has a subgraph $uv_1$. 
Thus every component of $G[W'\cup\{u\}]$ which does not contain $u$ must be a singleton. It follows that
\[
A\cap W=\{v_{2p-1},\,v_{2p-2},\,\dots,\,v_3\}
\quad\text{and}\quad
B=\{v_1\}.
\]
Therfore, this case gives $l$ possibilities for the vertex $u$,
which determines the partition $(A,B,C)$ as a consequence.
\item
The other case $u\in Y$ gives $l'$ possibilities for the partition $(A,B,C)$,
by symmetry.
\end{itemize}
In summary, 
\[
[s_\lambda]X_G=-2p+l+l'\le -2(l+1)+l+(l+1)<0.
\]
This contradiction proves $p\le l$.

Thirdly, we consider the graphs $G=\mathrm{br}'(2,\,2p-1,\,3)$ and will show that $p\le 6$.
Note that $G$ has independence number $p+4$.
Suppose to the contrary that $p\ge 7$.
For the partition $\lambda=(p+3)(p+1)1$,
the set $\mathcal{T}_{\lambda}$ consists of 3 tabloids, see \cref{fig:RHT:BS:l=2:l'=3}. 
\begin{figure}[htbp]
\begin{tikzpicture}[scale=0.5]
\input{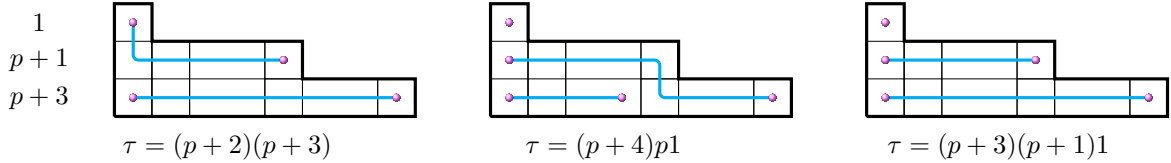}
\end{tikzpicture}
\caption{The special rim hook tabloids in the set $\mathcal{T}_{(p+3)(p+1) 1}$ for the graph~$\mathrm{br}'(2,\, 2p-1,\, 3)$.}
\label{fig:RHT:BS:l=2:l'=3}
\end{figure}

For the type $(p+3)(p+2)$, $G$ has a unique stable partition 
\[
(\{v_1,\,v_3,\,\dots,\,v_{2p-1}\}\cup Y,\ \{v_2,\,v_4,\,\dots,\,v_{2p}\}\cup X).
\]
This partition contributes $-1$ to $[s_\lambda]X_G$.
For stable partitions of types $(p+4)p1$ and $(p+3)(p+1)1$,
we can suppose that $(A,B,C)$ is such a partition with $\abs{A}>\abs{B}>\abs{C}$.

Suppose that $\abs{A}=p+4$. Then $\{v_1,v_{2p}\}\cap A=\emptyset$.
It follows that $\abs{A\cap W}\le p-1$ and 
\[
A=X\cup Y\cup\{v_2,\,v_4,\,\dots,\,v_{2j},\,v_{2j+3},\,v_{2j+5},\,\dots,\,v_{2p-1}\}
\quad\text{for some $j\in\{0,1,\dots,p-1\}$}.
\]
For each $j\in\{0,1,\dots,p-1\}$, the singleton $C$ is either $\{v_{2j+1}\}$ or $\{v_{2j+2}\}$.
Therefore, such partitions contribute $-2p$ to $[s_\lambda]X_G$.

Suppose that $\abs{A}=p+3$ and $\abs{B}=p+1$. 
Then $\abs{A\cap W}\in\{p-2,\,p-1,\,p\}$.
\begin{itemize}
\item
If $\abs{A\cap W}=p-2$, then $X\cup Y\subset A$.
It follows that $B\subset W$ and thus $\abs{B}\le p$, a contradiction.
\item
If $\abs{A\cap W}=p-1$, then there exists $u\in X\cup Y$ such that $X\cup Y\backslash A=\{u\}$,
and the set $B$ is contained in the subgraph $G[\{u\}\cup W]$.
\begin{itemize}
\item
Suppose that $u\in X$.
Since $B$ is stable and of order $p+1$, we derive that $B=\{u,\,v_2,\,v_4,\,\dots,\,v_{2p}\}$.
Since the vertex in $X\backslash \{u\}$ is in $A$, we find $v_1\not\in A$, and $C=\{v_1\}$.
Thus the contribution of the case $u\in X$ is 2 for $\abs{X}=2$. 
\item
For the same reason, the case $u\in Y$ contributes 3, for $\abs{Y}=3$. 
\end{itemize}
The total contribution is $5$.
\item
If $\abs{A\cap W}=p$, then $\{v_1,v_{2p}\}\cap A\ne\emptyset$. Since $\abs{A}=p+3$, we find
\[
A=\{v_1,\,v_3,\,\dots,\,v_{2p-1}\}\cup Y.
\]
Since the singleton $C$ may consist of any single vertex of the $(p+2)$-set $X\cup\{v_2,v_4,\dots,v_{2p}\}$,
the contribution of this case is $p+2$.
\end{itemize}
Summing up all contributions above, we obtain
\[
[s_\lambda]X_G=-1-2p+5+(p+2)=6-p<0.
\]
This contradiction proves that $p\le 6$.

Now, the remaining graphs in the family $\mathrm{br}'(l,p,l')$ are listed as follows.
\begin{itemize}
\item
$\mathrm{br}'(5,\,2p-1,\,l')$ with $l'\in\{5,6\}$ and $p\in[5]$.
\item
$\mathrm{br}'(4,\,2p-1,\,l')$ with $l'\in\{4,5\}$ and $p\in[4]$.
\item
$\mathrm{br}'(3,\,2p-1,\,l')$ with $l'\in\{3,4\}$ and $p\in[3]$.
\item
$\mathrm{br}'(2,\,2p-1,\,3)$ with $p\in[6]$.
\end{itemize}
We compute the chromatic symmetric function of each of these graphs by using mathematical software, and obtain the desired classification. 
\end{proof}

\begin{proposition}
For any integer $b\ge 1$, the graph $\mathrm{br}'(2,b,2)$ is not $e$-positive.
\end{proposition}
\begin{proof}
For even $b$, the graph $\mathrm{br}'(2,b,2)$ is not balanced and thus not Schur positive by \cref{thm:nS:balanced}. 
Let $G=\mathrm{br}'(2,\,2p-1,\,2)$ where $p\ge 1$.
We will show that $[e_{(2p+2)2}]X_G<0$.

Consider $G$ as consisting of the path $v_1\dotsm v_{2p}$
and the stars with edge sets $\{v_1x,\,v_1x'\}$ and $\{v_{2p}y,\,v_{2p}y'\}$.
Taking $e_1=v_2v_1$, $e_2=v_1x$ and $e_3=xv_2$ in \cref{thm:rec:3del}, we obtain
\begin{equation}\label{pf:csf:SPS:2b2}
X_G=e_1X_{G-x}+X_{\mathrm{br}(2p+1,\,2)}-2e_2X_{\mathrm{br}(2p-1,\,2)}.
\end{equation}
By using \cref{thm:rec:3del} in the same way, we obtain 
\begin{equation}\label{pf:csf:broom:2a-1:2}
X_{\mathrm{br}(2a-1,\,2)}
=e_1X_{P_{2a+1}}+X_{P_{2a+2}}-2e_2X_{P_{2a}},\quad\text{for $a\ge 1$}.
\end{equation}
These two relations are illustrated in \cref{fig:triple-deletion:2bridges},
in which the triple $(A,B,C)$ is $(\mathrm{br}(2p-2,2),\,K_2,\,K_1)$ for \cref{pf:csf:SPS:2b2} and $(P_{2a-2},\,K_2,\,K_1)$ for \cref{pf:csf:broom:2a-1:2}.
\begin{figure}[htbp]
\begin{center}
\begin{tikzpicture}[scale=.5]
\input{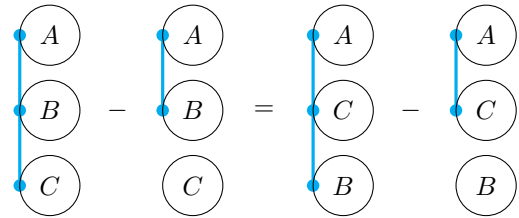}
\end{tikzpicture}
\caption{The triple-deletion rule for bridges $e_1$ and $e_2$.}\label{fig:triple-deletion:2bridges}
\end{center}
\end{figure}
Extracting the coefficients of $e_{2a+2}$ and $e_{2a}e_2$ from each side of \cref{pf:csf:broom:2a-1:2}, we can compute by \cref{prop:Wolfe} as
\begin{align*}
[e_{2a+2}]X_{\mathrm{br}(2a-1,\,2)}
&=[e_{2a+2}]X_{P_{2a+2}}=2a+2\quad\text{and}\\
[e_{2a}e_2]X_{\mathrm{br}(2a-1,\,2)}
&=[e_{2a}e_2]X_{P_{2a+2}}-2[e_{2a}]X_{P_{2a}}
=(6a-2)-2\cdotp 2a=2a-2.
\end{align*}
Extracting the coefficient of $e_{(2p+2)2}$ from both sides of \cref{pf:csf:SPS:2b2},
with the aid of the two fomulas above, we obtain
\[
[e_{(2p+2)2}]X_G
=[e_{2p+2}e_2]X_{\mathrm{br}(2p+1,\,2)}-2[e_{2p+2}]X_{\mathrm{br}(2p-1,\,2)}
=2p-2(2p+2)
=-2p-4
<0.
\]
This completes the proof.
\end{proof}

\begin{conjecture}\label{conj:Schur:SPS:22}
For any integer $p\ge 1$, the graph $\mathrm{br}'(2,2p-1,2)$ is Schur positive.
\end{conjecture}
We checked that \cref{conj:Schur:SPS:22} is true up to $2p-1=19$.

\bibliography{../csf}

\end{document}